\documentclass[10pt]{amsart}
\usepackage[T1]{fontenc}
\usepackage[latin5]{inputenc}
\usepackage{float}
\usepackage{amssymb}
\usepackage{amsfonts}
\usepackage[centertags]{amsmath}
\usepackage{amsthm}
\usepackage{newlfont}
 \makeatletter
\theoremstyle{plain}
\newtheorem{thm}{Theorem}[section]
\numberwithin{equation}{section}
\numberwithin{figure}{section}  
\theoremstyle{plain}
\newtheorem{pr}[thm]{Proposition}
\theoremstyle{plain}
\newtheorem{cor}[thm]{Corollary} 
\theoremstyle{plain}

\theoremstyle{plain}
\newtheorem{lem}[thm]{Lemma} 
\theoremstyle{plain}

\begin{document}
\title{On the fine spectrum of the forward difference
operator on the Hahn space}
\thanks{
**Corresponding author.}
\author{Med\.ine Ye\c s\.ilkayag\.il and Murat K\.ir\.i\c sc\.i**}
\address[Medine Ye\c silkayagil]{Department of Mathematics, U\c sak University, 1 Eyl\"{u}l Campus, 64200 - U\c{s}ak, Turkey}
\email{medine.yesilkayagil@usak.edu.tr}
\address[Murat Kiri\c s\c ci] {Department of Mathematical Education, Hasan Ali Yucel Education Faculty,
\.Istanbul University, Vefa, 34470--\.Istanbul, Turkey}
\email{mkirisci@hotmail.com, murat.kirisci@istanbul.edu.tr}
\begin{abstract}
The main purpose of this paper is to determine the fine spectrum
with respect to Goldberg's classification of the difference operator
over the sequence space $h$. As a new development, we give the
approximate point spectrum, defect spectrum and compression spectrum
of the difference operator on the sequence space $h$.
\end{abstract}
\maketitle
\section{Introduction}
An important branch of mathematics due to its application in
other branches of science is a Spectral Theory.
It has been proved to be a very useful tool because of its
convenient and easy applicability of the different fields.
In numerical analysis, the spectral values may determine whether a discretization
of a differential equation will get the right answer or how fast a conjugate gradient
iteration will converge. In aeronautics, the spectral values may determine whether
the flow over a wing is laminar or turbulent. In electrical engineering, it may determine
the frequency response of an amplifier or the reliability of a power system. In
quantum mechanics, it may determine atomic energy levels and thus, the frequency
of a laser or the spectral signature of a star. In structural mechanics, it may determine
whether an automobile is too noisy or whether a building will collapse in
an earthquake. In ecology, the spectral values may determine whether a food web
will settle into a steady equilibrium. In probability theory, they may determine the
rate of convergence of a Markov process.

There are several studies about the spectrum of the linear
operators defined by some triangle matrices over certain sequence
spaces. This long-time behavior was intensively studied over many years,
starting with the work by Wenger \cite{wen}, who established the
fine spectrum of the integer power of the Ces\`{a}ro operator in c.
The generalization of \cite{wen} to the weighted mean methods
is due to Rhoades \cite{rho}. The study of the fine spectrum of the  operator
on the sequence space $\ell_{p}$, $(1< p <\infty)$ was initiated by
Gonz\`{a}lez \cite{gon}. The method of the spectrum of
the Ces\`{a}ro operator  prepared by Reade \cite{rea}, Akhmedov and
Ba\c sar \cite{aafb}, and Okutoyi \cite{oku}, respectively, whose established to this idea on
the sequence spaces $c_{0}$ and $bv$.In \cite{yil}, the fine spectrum of the Rhaly operators
on the sequence spaces $c_{0}$ and $c$ was given. The spectrum and fine spectrum for
p-Ces\`{a}ro operator acting on the space $c_{0}$ was studied by Co\c skun \cite{cc}.
The investigation of the spectrum and the fine spectrum of the difference
operator on the sequence spaces $s_{r}$ and $c_{0}$, $c$  was made by
Malafosse \cite{mala} and Altay and Ba\c sar \cite{bafb1}, respectively, where
$s_{r}$ denotes the Banach space of all sequences $x =(x_{k})$ normed by
$\|x\|_{s_r}=\underset{k\in\mathbb{N}}{\sup}|x_k|/r^k$, $(r>0)$. The idea of the
fine spectrum applied to the Zweier matrix which is a band matrix as an operator
over the sequence spaces $\ell_1$ and $bv$ by Altay and Karakus \cite{bamk}.
Let $\Delta_{\nu}$ is double sequential band matrix on $\ell_{1}$ such that
$(\Delta_{\nu})_{nn}=\nu_n$ and $(\Delta_{\nu})_{n+1,n}=-\nu_n$ for all
$n\in\mathbb N$, under certain conditions on the sequence $\nu =(\nu_{k})$.
The spectra and the fine spectra of matrix $\Delta_{\nu}$ were determined by Srivastava and
Kumar \cite{srsk}. Afterwards, these results of the double sequential band matrix
$\Delta_{\nu}$ generalized to the double sequential band matrix
$\Delta_{uv}$ such that defined by
$\Delta_{uv}x=(u_{n}x_{n}+v_{n-1}x_{n-1})_{n\in\mathbb N}$ for all
$n\in\mathbb N$ (see \cite{Sr11}). In \cite{mal}, the
fine spectra of the Toeplitz operators represented by an upper and
lower triangular $n$-band infinite matrices, over the sequence
spaces $c_0$ and $c$ was computed. The fine
spectra of upper triangular double-band matrices over the sequence
spaces $c_0$ and $c$ was obtained by Karakaya and Altun \cite{vkma}.
Let  $\Delta_{a,b}$ is a double band matrix with the convergent sequences
$\widetilde{a}=(a_k)$ and $\widetilde{b}=(b_k)$ having certain properties,
over the sequence space $c$. The fine spectrum of the matrix $\Delta_{a,b}$
was examined by Akhmedov and El-Shabrawy \cite{aafb3}. The approach to the
fine spectrum with respect to Goldberg's classification studied with of the
operator $B(r,s,t)$ defined by a triple band matrix over the sequence spaces
$\ell_{p}$ and $bv_p$, $(1<p<\infty)$ by Furkan et al. \cite{Fu10}. Quite recently,
the fine spectrum with respect to Goldberg's classification of the operator defined
by the lambda matrix over the sequence spaces $c_{0}$ and $c$ computed by
Ye\c{s}ilkayagil and Ba\c{s}ar \cite{medine}.

Hahn sequence space is defined as $x=(x_{k})$ such that $\sum_{k=1}^{\infty}k|x_{k}-x_{k+1}|$ converges and $x_{k}$ is a null sequence and is denoted by $h$. Initially, this space was defined and studied to some general properties by Hahn \cite{Hahn}. It was examined different properties of this space by Goes and Goes \cite{GoesII} and Rao \cite{Rao}, \cite{RaoI}, \cite{RaoII}. Quite recently, the studies on Hahn sequence space has been compiled by Kirisci \cite{kirisci1}. Also in \cite{kirisci2}, it has been defined a new Hahn sequence space by Ces\`{a}ro mean.

In the present paper, our propose is to investigate the fine spectrum of the difference operator $\Delta$ on the sequence space $h$.
And also, we define the approximate point spectrum, defect spectrum and compression spectrum
of the difference operator on the sequence space $h$, as a new approach.

\section{Preliminaries and Definition}
Let $X$ and $Y$ be Banach spaces, and also let $T:X\rightarrow Y$ be
a bounded linear operator. The range of the operator $T$ defined by
\begin{eqnarray*}
R(T)=\{ y\in Y:y=Tx,\text{ }x\in X\}.
\end{eqnarray*}
The \textit{set of all bounded linear operators} on $X$ into itself
denoted by $B(X)$.

We choose any Banach space $X$ and let $T\in B(X) $. Then we can define
the \textit{adjoint} $T^{\ast }$ of $T$ is a bounded
linear operator on the dual $X^{\ast }$ of $X$ such that $\left(
T^{\ast }f\right)(x)=f(Tx)$ for all $f\in Y^{\ast }$ and $x\in X$.

Let $X\neq \{\theta\}$ be a non-trivial complex normed space. A linear operator
$T$ defined by $T:D(T) \rightarrow X$ on a subspace $D(T) \subseteq X$.
We do not assume that $D(T)$ is dense in $X$ or that T
has closed graph $\{(x,Tx) : x \in D(T)\} \subseteq X \times X$. We
mean by the expression \textit{ ''$T$ is invertible''} that there
exists a bounded linear operator $S : R(T) \rightarrow X$ for which
$ST=I$ on $D(T)$ and $\overline{R(T)}=X$; such that $S=T^{-1}$ is
necessarily uniquely determined, and linear; the boundedness of $S$
means that $T$ must be \textit{bounded below}, in the sense that
there is $k >0$ for which $\|Tx\| \geq k\|x\|$ for all $x\in D(T)$.
The perturbed operator defined on the same domain $D(T)$ as $T$ as follows:
\begin{eqnarray*}
T_{\alpha}=\alpha I-T
\end{eqnarray*}
such that associated with each complex number $\alpha$. The \textit{spectrum}
$\sigma(T,X)$ consists of those $\alpha\in\mathbb{C}$ for which
$T_{\alpha}$ is not invertible, and the \textit{resolvent} is the
mapping from the complement $\sigma(T,X)$ of the spectrum into the
algebra of bounded linear operators on X defined by $\alpha\mapsto
T_{\alpha}^{-1}$.

Let $\omega$ is the space of all complex valued sequences and $\phi$
the set of all infinitely nonzero sequences. A linear subspace of $\omega$
which contain $\phi$ said a \textit{sequence space}. We write $\ell_\infty$,
$c$, $c_0$ and $bv$ for the spaces of all bounded, convergent, null
and bounded variation sequences which are the Banach spaces with the
sup-norm $\|x\|_{\infty}=\underset{k\in\mathbb N}{\sup}|x_k|$ and
$\|x\|_{bv}=\stackrel{\infty}{\underset{k=0}{\sum}}|x_k-x_{k+1}|$,
respectively, while $\phi$ is not a Banach space with respect to any
norm, where $\mathbb{N}=\{0,1,2,\ldots\}$. Also by $\ell_p$, we
denote the space of all $p$--absolutely summable sequences which is
a Banach space with the norm
$\|x\|_p=\left(\stackrel{\infty}{\underset{k=0}{\sum}}|x_k|^p\right)^{1/p}$,
where $1\leq p<\infty$.

Let $\mu$ and $\nu$ be two sequence spaces, and $A=(a_{nk})$ be an
infinite matrix of complex numbers $a_{nk}$, where
$k,n\in\mathbb{N}$. Then, we say that $A$ defines a matrix mapping
from $\mu$ into $\nu$, and we denote it by writing
$A:\mu\rightarrow\nu$ if for every sequence $x=(x_{k})\in\mu$, the
$A$-transform $Ax =\{(Ax)_{n}\}$ of $x$ is in $\nu$; where
\begin{eqnarray}\label{eq04}
(Ax)_{n}=\sum_{k=0}^{\infty}a_{nk}x_{k}~\textrm{ for all
}~n\in\mathbb{N}.
\end{eqnarray}
By $(\mu:\nu)$, we denote the class of all matrices $A$ such that $A
:\mu\rightarrow\nu$. Thus, $A\in(\mu:\nu)$ if and only if the series
on the right side of (\ref{eq04}) converges for each
$n\in\mathbb{N}$ and each $x\in\mu$, and we have $Ax=\{(Ax)_{n}\}_{n
\in\mathbb{N}}\in\nu$ for all $x\in\mu$.

The $BK-$space $h$ of all sequences $x=(x_{k})$ such that
\begin{eqnarray*}
h=\left \{x: \sum_{k=1}^{\infty} k|\Delta x_{k}|<\infty  ~\textrm{ and }~  \lim_{k \rightarrow \infty}x_{k}=0 \right\}
\end{eqnarray*}
was defined by Hahn \cite{Hahn}. Here and after $\Delta$ denotes the forward difference operator, that is, $\Delta x_{k}=x_{k}-x_{k+1}$, for all $k\in \mathbb{N}$. The following norm
\begin{eqnarray*}
\|x\|_{h}=\sum_{k} k|\Delta x_{k}|+\sup_{k}|x_{k}|
\end{eqnarray*}
was given on the space $h$ by Hahn \cite{Hahn} (and also \cite{GoesII}).
Rao \cite[Proposition 2.1]{Rao} defined a new norm on $h$ as $\|x\|=\sum_{k} k|\Delta x_{k}|.$

Hahn proved following properties of the space $h$:
\begin{lem}\label{lem1}
The following statements hold:
\begin{itemize}
\item[(i)] $h$ is a Banach space.
\item[(ii)]$h \subset \ell_{1} \cap \int c_{0}$.
\item[(iii)] $h^{\beta}=\rho_{\infty}$,
\end{itemize}
where $\int \lambda=\big\{x=(x_{k})\in \omega: (kx_{k})\in \lambda \big\}$ and $\rho_{\infty}=\big\{x=(x_{k})\in \omega: \sup_{n}n^{-1}\big|\sum_{k=1}^{n}x_{k}\big| <\infty \big\}$.
\end{lem}

Functional analytic properties of the $BK-$space $bv_{0}\cap
d\ell_{1}$ was studied by Goes and Goes \cite{GoesII}, where
$d\ell_{1}=\{x=(x_k)\in
\omega:\sum_{k=1}^{\infty}\frac{1}{k}|x_k|<\infty\}$. Also, in
\cite{GoesII}, the arithmetic means of sequences in $bv_{0}$ and
$bv_{0}\cap d\ell_{1}$ were considered, and used the fact that the
Ces\`{a}ro transform $(n^{-1}\sum_{k=1}^{n}x_{k})$ of order one
$x\in bv_{0}$ is a quasiconvex null sequence.

Rao \cite{Rao} studied some geometric properties of Hahn sequence space and gave the
characterizations of some classes of matrix transformations. Also, in \cite{RaoI} and
\cite{RaoII}, Rao examined the different properties of Hahn sequence space.

Balasubramanian and Pandiarani \cite{balasub} defined the new sequence space $h(F)$
called the Hahn sequence space of fuzzy numbers and proved that $\beta-$ and
$\gamma-$duals of $h(F)$ is the Ces\`{a}ro space of the set of all fuzzy bounded sequences.

Until the new studies of Kiri\c s\c ci \cite{kirisci1,kirisci2}, there has not been
any work containing the Hahn sequence space.

\section{Subdivision of the Spectrum}
In this section, we define the parts called point
spectrum, continuous spectrum, residual spectrum, approximate point
spectrum, defect spectrum and compression spectrum of the spectrum.
There are many different ways to subdivide the spectrum of a bounded
linear operator. Some of them are motivated by applications to
physics, in particular, quantum mechanics.
\subsection{The Point Spectrum, Continuous Spectrum and Residual Spectrum} The name resolvent is appropriate since $T_{\alpha}^{-1}$ helps to solve the equation $T_{\alpha}x=y$. Thus,
$x=T_{\alpha}^{-1}y$ provided that $T_{\alpha}^{-1}$ exists. More
importantly, the investigation of properties of $T_{\alpha}^{-1}$
will be basic for an understanding of the operator $T$ itself.
Naturally, many properties of $T_{\alpha}$ and $T_{\alpha}^{-1}$
depend on $\alpha$, and the spectral theory is concerned with those
properties. For instance, we shall be interested in the set of all
$\alpha$'s in the complex plane such that $T_{\alpha}^{-1}$ exists.
Boundedness of $T_{\alpha}^{-1}$ is another property that will be
essential. We shall also ask for what $\alpha$'s the domain of
$T_{\alpha}^{-1}$ is dense in X, to name just a few aspects. A
\textit{regular value} $\alpha$ of $T$ is a complex number such that
$T_{\alpha}^{-1}$ exists and is bounded whose domain is dense in X.
For our investigation of T, $T_{\alpha}$ and $T_{\alpha}^{-1}$, we
need some basic concepts in the spectral theory which are given, as
follows (see  Kreyszig \cite[pp. 370-371]{ek}):

The \textit{resolvent set} $\rho (T,X)$ of $T$ is the set of all
regular values $\alpha$ of $T$.  Furthermore, the spectrum $\sigma
(T,X)$ is partitioned into the following three disjoint sets:

The \textit{point (discrete) spectrum} $\sigma_p(T,X)$ is the set
such that $T_\alpha^{-1}$ does not exist. An
$\alpha\in\sigma_p(T,X)$ is called an \textit{eigenvalue} of $T$.

The \textit{continuous spectrum} $\sigma_c(T,X)$ is the set such
that $T_\alpha^{-1}$ exists and is unbounded, and the domain of
$T_\alpha^{-1}$ is dense in $X$.

The \textit{residual spectrum} $\sigma_r(T,X)$ is the set such that
$T_\alpha^{-1}$ exists (and may be bounded or not) but the domain of
$T_\alpha^{-1}$ is not dense in $X$.

Therefore, these three subspectra form a disjoint subdivision such
that
\begin{eqnarray}\label{eq1}
\sigma (T,X)=\sigma_{p}(T,X)\cup \sigma_{c}(T,X)\cup \sigma
_{r}(T,X).
\end{eqnarray}To avoid trivial misunderstandings, let us say that some of the sets defined above may
be empty. This is an existence problem which we shall have to
discuss. Indeed, it is well-known that
$\sigma_{c}(T,X)=\sigma_{r}(T,X)=\emptyset$ and the spectrum
$\sigma(T,X)$ consists of only the set $\sigma_{p}(T,X)$ in the
finite-dimensional case.
\subsection{The Approximate Point Spectrum, Defect Spectrum and Compression Spectrum}
In this subsection, three more subdivision of the spectrum called the \textit{approximate point spectrum},
\textit{defect spectrum} and \textit{compression spectrum} have been defined as in Appell et al. \cite{ap}.

Let $X$ is a Banach space and $T$ is a bounded linear operator. A $(x_{k}) \in X$ \textit{Weyl sequence}
for $T$ defined by $\left\Vert x_{k}\right\Vert =1$ and $\left\Vert Tx_{k}\right\Vert
\rightarrow 0$, as $ k\rightarrow \infty$.

In what follows, we call the set
\begin{eqnarray}\label{eq2}
\sigma_{ap}(T,X):=\{\alpha \in \mathbb{C}:\text{there exists a Weyl
sequence for }\alpha I-T\}
\end{eqnarray}
the \textit{approximate point spectrum of $T$}. Moreover, the
subspectrum
\begin{eqnarray}\label{eq30}
\sigma_{\delta }(T,X):=\{\alpha \in \mathbb{C}:\alpha I-T\text{ is
not surjective}\}
\end{eqnarray}
is called \textit{defect spectrum} of $T$.

The two subspectra given by (\ref{eq2}) and (\ref{eq30}) form a (not
necessarily disjoint) subdivision
\begin{eqnarray*}
\sigma (T,X)=\sigma_{ap}(T,X)\cup \sigma_{\delta }(T,X)
\end{eqnarray*}
of the spectrum. There is another subspectrum,
\begin{eqnarray*}
\sigma_{co}(T,X)=\{\alpha \in \mathbb{C}:\overline{R(\alpha
I-T)}\neq X\}
\end{eqnarray*}
which is often called \textit{compression spectrum} in the
literature. The compression spectrum gives rise to another (not
necessarily disjoint) decomposition
\begin{eqnarray*}
\sigma (T,X)=\sigma_{ap}(T,X)\cup \sigma_{co}(T,X)
\end{eqnarray*}
of the spectrum. Clearly, $\sigma_{p}(T,X)\subseteq
\sigma_{ap}(T,X)$ and $ \sigma_{co}(T,X)\subseteq \sigma_{\delta
}(T,X)$. Moreover, comparing these subspectra with those in
(\ref{eq1}) we note that
\begin{eqnarray*}
\sigma_{r}(T,X)&=&\sigma_{co}(T,X)\backslash \sigma_{p}(T,X),\\
\sigma_{c}(T,X)&=&\sigma (T,X)\backslash \lbrack \sigma_{p}(T,X)\cup
\sigma_{co}(T,X)\rbrack.
\end{eqnarray*}

Sometimes it is useful to relate the spectrum of a bounded linear
operator to that of its adjoint. Building on classical existence and
uniqueness results for linear operator equations in Banach spaces
and their adjoints are also useful.

\begin{pr} \cite[Proposition 1.3, p. 28]{ap}\label{pr21}
The following relations on the spectrum and subspectrum of an operator $T\in B(X)$ and its adjoint
$T^{\ast}\in B(X^{\ast })$ hold:
\begin{enumerate}
 \item[(a)] $\sigma(T^{\ast},X^{\ast})=\sigma (T,X)$.
  \item[(b)] $\sigma_{c}(T^{\ast},X^{\ast})\subseteq \sigma_{ap}(T,X)$.
  \item[(c)] $\sigma_{ap}(T^{\ast},X^{\ast})=\sigma_{\delta }(T,X)$.
  \item[(d)] $ \sigma_{\delta}(T^{\ast},X^{\ast})=\sigma_{ap}(T,X)$.
  \item[(e)] $\sigma_{p}(T^{\ast},X^{\ast})=\sigma_{co}(T,X)$.
  \item[(f)] $\sigma_{co}(T^{\ast},X^{\ast})\supseteq \sigma_{p}(T,X)$.
  \item[(g)] $\sigma(T,X)=\sigma_{ap}(T,X)\cup\sigma_{p}(T^{\ast},X^{\ast})=\sigma_{p}(T,X)\cup\sigma_{ap}(T^{\ast},X^{\ast})$.\end{enumerate}
\end{pr}
The relations (c)-(f) show that the approximate point spectrum is in
a certain sense dual to the defect spectrum and the point spectrum
is dual to the compression spectrum. The equality (g) implies, in
particular, that $\sigma (T,X)=\sigma_{ap}(T,X)$ if $X$ is a Hilbert
space and $T$ is normal. Roughly speaking, this shows that normal
(in particular, self-adjoint) operators on Hilbert spaces are most
similar to matrices in finite dimensional spaces (see Appell et al.
\cite{ap}).
\subsection{Goldberg's Classification of Spectrum}

If $X$ is a Banach space and $T\in B(X)$, then there are three
possibilities for $R(T)$:
\begin{itemize}
\item [(A)] $\quad R(T)=X$.
\item [(B)] $\quad R(T)\neq\overline{R(T)}=X$.
\item [(C)] $\quad \overline{R(T)}\neq X$.\end{itemize}
and
\begin{itemize}
\item [(1)] $\quad T^{-1}$ exists and is continuous.
\item [(2)] $\quad T^{-1}$ exists but is discontinuous.
\item [(3)] $\quad T^{-1}$ does not exist.
\end{itemize}

If these possibilities are combined in all possible ways, nine
different states are created. These are labelled by: $A_{1}$,
$A_{2}$, $A_{3}$, $ B_{1} $, $B_{2}$, $B_{3}$, $C_{1}$, $C_{2}$,
$C_{3}$. If an operator is in state $C_{2}$ for example, then
$\overline{R(T)} \neq X$ and $T^{-1}$ exists but is discontinuous
(see Goldberg \cite{go}).

\begin{picture}(300,300)(0,0)

\put(100,290){\line(1,0){180}} \put(100,270){\line(1,0){180}}
\put(100,250){\line(1,0){180}} \put(100,230){\line(1,0){180}}
\put(100,210){\line(1,0){180}} \put(100,190){\line(1,0){180}}
\put(100,170){\line(1,0){180}} \put(100,150){\line(1,0){180}}
\put(100,130){\line(1,0){180}} \put(100,110){\line(1,0){180}}

\put(100,290){\line(1,-1){20}} \put(100,270){\line(1,1){20}}
\put(100,270){\line(1,-1){20}} \put(100,250){\line(1,1){20}}
\put(100,250){\line(1,-1){20}} \put(100,230){\line(1,1){20}}
\put(100,230){\line(1,-1){20}} \put(100,210){\line(1,1){20}}
\put(100,210){\line(1,-1){20}} \put(100,190){\line(1,1){20}}
\put(100,190){\line(1,-1){20}} \put(100,170){\line(1,1){20}}
\put(100,170){\line(1,-1){20}} \put(100,150){\line(1,1){20}}
\put(100,150){\line(1,-1){20}} \put(100,130){\line(1,1){20}}

\put(120,290){\line(1,-1){20}} \put(120,270){\line(1,1){20}}
\put(120,270){\line(1,-1){20}} \put(120,250){\line(1,1){20}}
\put(120,250){\line(1,-1){20}} \put(120,230){\line(1,1){20}}
\put(120,230){\line(1,-1){20}} \put(120,210){\line(1,1){20}}
\put(120,210){\line(1,-1){20}} \put(120,190){\line(1,1){20}}
\put(120,190){\line(1,-1){20}} \put(120,170){\line(1,1){20}}
\put(120,170){\line(1,-1){20}} \put(120,150){\line(1,1){20}}
\put(120,150){\line(1,-1){20}} \put(120,130){\line(1,1){20}}
\put(120,130){\line(1,-1){20}} \put(120,110){\line(1,1){20}}

\put(140,290){\line(1,-1){20}} \put(140,270){\line(1,1){20}}
\put(140,270){\line(1,-1){20}} \put(140,250){\line(1,1){20}}
\put(140,230){\line(1,-1){20}} \put(140,210){\line(1,1){20}}
\put(140,210){\line(1,-1){20}} \put(140,190){\line(1,1){20}}
\put(140,190){\line(1,-1){20}} \put(140,170){\line(1,1){20}}
\put(140,170){\line(1,-1){20}} \put(140,150){\line(1,1){20}}
\put(140,150){\line(1,-1){20}} \put(140,130){\line(1,1){20}}
\put(140,130){\line(1,-1){20}} \put(140,110){\line(1,1){20}}

\put(160,290){\line(1,-1){20}} \put(160,270){\line(1,1){20}}
\put(160,270){\line(1,-1){20}} \put(160,250){\line(1,1){20}}
\put(160,250){\line(1,-1){20}} \put(160,230){\line(1,1){20}}
\put(160,230){\line(1,-1){20}} \put(160,210){\line(1,1){20}}
\put(160,210){\line(1,-1){20}} \put(160,190){\line(1,1){20}}
\put(160,190){\line(1,-1){20}} \put(160,170){\line(1,1){20}}
\put(160,170){\line(1,-1){20}} \put(160,150){\line(1,1){20}}
\put(160,150){\line(1,-1){20}} \put(160,130){\line(1,1){20}}
\put(160,130){\line(1,-1){20}} \put(160,110){\line(1,1){20}}

\put(180,290){\line(1,-1){20}} \put(180,270){\line(1,1){20}}
\put(180,270){\line(1,-1){20}} \put(180,250){\line(1,1){20}}
\put(180,250){\line(1,-1){20}} \put(180,230){\line(1,1){20}}
\put(180,230){\line(1,-1){20}} \put(180,210){\line(1,1){20}}
\put(180,210){\line(1,-1){20}} \put(180,190){\line(1,1){20}}
\put(180,190){\line(1,-1){20}} \put(180,170){\line(1,1){20}}
\put(180,170){\line(1,-1){20}} \put(180,150){\line(1,1){20}}
\put(180,150){\line(1,-1){20}} \put(180,130){\line(1,1){20}}
\put(180,130){\line(1,-1){20}} \put(180,110){\line(1,1){20}}

\put(200,290){\line(1,-1){20}} \put(200,270){\line(1,1){20}}
\put(200,270){\line(1,-1){20}} \put(200,250){\line(1,1){20}}
\put(200,250){\line(1,-1){20}} \put(200,230){\line(1,1){20}}
\put(200,230){\line(1,-1){20}} \put(200,210){\line(1,1){20}}
\put(200,210){\line(1,-1){20}} \put(200,190){\line(1,1){20}}
\put(200,190){\line(1,-1){20}} \put(200,170){\line(1,1){20}}
\put(200,170){\line(1,-1){20}} \put(200,150){\line(1,1){20}}
\put(200,150){\line(1,-1){20}} \put(200,130){\line(1,1){20}}
\put(200,130){\line(1,-1){20}} \put(200,110){\line(1,1){20}}

\put(220,290){\line(1,-1){20}} \put(220,270){\line(1,1){20}}
\put(220,270){\line(1,-1){20}} \put(220,250){\line(1,1){20}}
\put(220,250){\line(1,-1){20}} \put(220,230){\line(1,1){20}}
\put(220,230){\line(1,-1){20}} \put(220,210){\line(1,1){20}}
\put(220,210){\line(1,-1){20}} \put(220,190){\line(1,1){20}}
\put(220,190){\line(1,-1){20}} \put(220,170){\line(1,1){20}}
\put(220,170){\line(1,-1){20}} \put(220,150){\line(1,1){20}}
\put(220,150){\line(1,-1){20}} \put(220,130){\line(1,1){20}}
\put(220,130){\line(1,-1){20}} \put(220,110){\line(1,1){20}}

\put(240,290){\line(1,-1){20}} \put(240,270){\line(1,1){20}}
\put(240,270){\line(1,-1){20}} \put(240,250){\line(1,1){20}}
\put(240,250){\line(1,-1){20}} \put(240,230){\line(1,1){20}}
\put(240,230){\line(1,-1){20}} \put(240,210){\line(1,1){20}}
\put(240,210){\line(1,-1){20}} \put(240,190){\line(1,1){20}}
\put(240,190){\line(1,-1){20}} \put(240,170){\line(1,1){20}}
\put(240,170){\line(1,-1){20}} \put(240,150){\line(1,1){20}}
\put(240,150){\line(1,-1){20}} \put(240,130){\line(1,1){20}}
\put(240,130){\line(1,-1){20}} \put(240,110){\line(1,1){20}}

\put(260,290){\line(1,-1){20}} \put(260,270){\line(1,1){20}}
\put(260,270){\line(1,-1){20}} \put(260,250){\line(1,1){20}}
\put(260,250){\line(1,-1){20}} \put(260,230){\line(1,1){20}}
\put(260,230){\line(1,-1){20}} \put(260,210){\line(1,1){20}}
\put(260,210){\line(1,-1){20}} \put(260,190){\line(1,1){20}}
\put(260,190){\line(1,-1){20}} \put(260,170){\line(1,1){20}}
\put(260,170){\line(1,-1){20}} \put(260,150){\line(1,1){20}}
\put(260,150){\line(1,-1){20}} \put(260,130){\line(1,1){20}}
\put(260,130){\line(1,-1){20}} \put(260,110){\line(1,1){20}}

\put(90,278){\makebox(0,0){$C_{3}$}}
\put(90,258){\makebox(0,0){$C_{2}$}}
\put(90,238){\makebox(0,0){$C_{1}$}}
\put(90,218){\makebox(0,0){$B_{3}$}}
\put(90,198){\makebox(0,0){$B_{2}$}}
\put(90,178){\makebox(0,0){$B_{1}$}}
\put(90,158){\makebox(0,0){$A_{3}$}}
\put(90,138){\makebox(0,0){$A_{2}$}}
\put(90,118){\makebox(0,0){$A_{1}$}}

\put(271,100){\makebox(0,0){$C_{3}$}}
\put(251,100){\makebox(0,0){$C_{2}$}}
\put(231,100){\makebox(0,0){$C_{1}$}}
\put(211,100){\makebox(0,0){$B_{3}$}}
\put(191,100){\makebox(0,0){$B_{2}$}}
\put(171,100){\makebox(0,0){$B_{1}$}}
\put(151,100){\makebox(0,0){$A_{3}$}}
\put(131,100){\makebox(0,0){$A_{2}$}}
\put(111,100){\makebox(0,0){$A_{1}$}}

\put(100,80){\vector(1,0){80}} \put(70,110){\vector(0,1){80}}
\put(140,70){\makebox(0,0){$T$}}
\put(60,150){\makebox(0,0){$T^{*}$}}

\put(100,290){\line(0,-1){180}} \put(120,290){\line(0,-1){180}}
\put(140,290){\line(0,-1){180}} \put(160,290){\line(0,-1){180}}
\put(180,290){\line(0,-1){180}} \put(200,290){\line(0,-1){180}}
\put(220,290){\line(0,-1){180}} \put(240,290){\line(0,-1){180}}
\put(260,290){\line(0,-1){180}} \put(280,290){\line(0,-1){180}}

\put(195,40){\makebox(0,0){ \small{Table 1.1: State diagram for
$B(X)$ and $B(X^{\ast })$ for a non-reflective Banach space $X$
}}}\end{picture}

If $\alpha $ is a complex number such that $T_\alpha\in A_{1}$ or
$T_\alpha\in B_{1}$, then $\alpha \in \rho (T,X)$. All scalar values
of $\alpha $ not in $\rho (T,X)$ comprise the spectrum of $T$. The
further classification of $\sigma (T,X)$ gives rise to the fine
spectrum of $T$. That is, $\sigma (T,X)$ can be divided into the
subsets $A_{2}\sigma (T,X)=\emptyset $, $A_{3}\sigma (T,X)$,
$B_{2}\sigma (T,X)$, $B_{3}\sigma (T,X)$, $C_{1}\sigma (T,X)$,
$C_{2}\sigma (T,X)$, $C_{3}\sigma (T,X)$. For example, if $T_\alpha$
is in a given state, $C_{2}$ (say), then we write $\alpha \in
C_{2}\sigma (T,X)$.

By the definitions given above, we can illustrate the subdivision
(\ref{eq1}) in the following table:
\begin{center}
\begin{tabular}{|c|c|c|c|c|}
\hline &  & 1 & 2 & 3 \\ \hline
&  & $T^{-1}_\alpha$ exists & $T^{-1}_\alpha$ exists & $T^{-1}_\alpha$ \\
&  & and is bounded & and is unbounded & does not exist \\
\hline\hline
&  &  &  & $\alpha \in \sigma_{p}(T,X)$ \\
A & $R(\alpha I-T)=X$ & $\alpha \in \rho (T,X)$ & -- & $\alpha
\in \sigma_{ap}(T,X)$ \\
&  &  &  &  \\ \hline\hline
&  &  & $\alpha \in \sigma_{c}(T,X)$ & $\alpha \in \sigma_{p}(T,X)$ \\
B & $\overline{R(\alpha I-T)}=X$ & $\alpha \in \rho (T,X)$ & $\alpha
\in
\sigma_{ap}(T,X)$ & $\alpha \in \sigma_{ap}(T,X)$ \\
&  &  & $\alpha \in \sigma_{\delta }(T,X)$ & $\alpha \in \sigma
_{\delta }(T,X)$ \\ \hline\hline &  & $\alpha \in \sigma_{r}(T,X)$ &
$\alpha \in \sigma_{r}(T,X)$ & $\alpha
\in \sigma_{p}(T,X)$ \\
C & $\overline{R(\alpha I-T)}\not=X$ & $\alpha \in \sigma _{\delta
}(T,X)$
& $\alpha \in \sigma_{ap}(T,X)$ & $\alpha \in \sigma_{ap}(T,X)$ \\
&  &  & $\alpha \in \sigma_{\delta }(T,X)$ & $\alpha \in \sigma
_{\delta
}(T,X)$ \\
&  & $\alpha \in \sigma_{co}(T,X)$ & $\alpha \in \sigma _{co}(T,X)$
& $ \alpha \in \sigma_{co}(T,X)$ \\ \hline
\end{tabular}

\vspace{0.1cm}Table 1.2: Subdivision of spectrum of a linear
operator
\end{center}

One can observe by the closed graph theorem that in the case $A_{2}$
cannot occur in a Banach space $X$. If we are not in the third
column of Table 1.2, i.e., if $\alpha $ is not an eigenvalue of $T$,
we may always consider the resolvent operator $T^{-1}_\alpha$ (on a
possibly \textit{thin} domain of definition) as \textit{algebraic
inverse} of $\alpha I-T$.

The forward difference operator $\Delta$ is represented by the matrix
\begin{eqnarray*}
\Delta= \left[ \begin{array}{cccccc}
1 & -1 & 0 &0&\ldots \\
0&1 & -1&0& \ldots \\
0&0 & 1 &-1 &  \ldots \\0&0 & 0&1 &  \ldots \\
\vdots & \vdots &  \vdots&  \vdots& \ddots
\end{array} \right].
\end{eqnarray*}

\begin{cor}
$\Delta :h\rightarrow h$ is a bounded linear operator.
\end{cor}

\section{On the fine spectrum of the forward difference operator on the Hahn space}
In this section, we determine the spectrum and fine spectrum of the forward difference operator $\Delta$ on the Hahn space $h$ and calculate the norm of the operator
\begin{thm} \label{Thm1}
$\sigma(\Delta,h)=\left\{\alpha
\in\mathbb{C}:\left|1-\alpha\right|\leq 1\right\}$.
\end{thm}
\begin{proof}
Let $\left|1-\alpha\right|>1$. Since $\Delta - \alpha I $ is triangle, $(\Delta - \alpha I )^{-1}$ exists and solving the matrix equation $(\Delta - \alpha I )x= y$  for $x$ in terms of $y$ gives the matrix $(\Delta - \alpha I )^{-1}=B=(b_{nk})$, where
\begin{eqnarray*}
b_{nk}= \left\{ \begin{array}{ccl}
\frac{1}{(1-\alpha)^{n+1}}&, & 0 \leq k \leq n,\\
0&, & k>n
\end{array} \right.
\end{eqnarray*}
for all $k,n\in\mathbb{N}$. Thus, we observe that
\begin{eqnarray*}
\|(\Delta - \alpha I )^{-1}\|_{(h:h)}&=&\sum_{n=1}^{\infty}n|b_{nk}-b_{n+1,k}|\nonumber\\
&\leq&\sum_{n=1}^{\infty}n|b_{nk}|+\sum_{n=1}^{\infty}n|b_{n+1,k}|\nonumber\\
&=&\sum_{n=1}^{\infty}\frac{n}{|1-\alpha|^{n+1}}+\sum_{n=1}^{\infty}\frac{n}{|1-\alpha|^{n+2}}.
\end{eqnarray*}
From the ratio test, we have
\begin{eqnarray*}
\|(\Delta - \alpha I )^{-1}\|_{(h:h)}<\infty,
\end{eqnarray*}
that is, $(\Delta - \alpha I )^{-1}\in (h:h)$. But for $\left|1-\alpha\right|\leq 1$,
\begin{eqnarray*}
\|(\Delta - \alpha I )^{-1}\|_{(h:h)}=\infty,
\end{eqnarray*}
that is, $(\Delta - \alpha I )^{-1}$ is not in $B(h)$.
This completes the proof.
\end{proof}

\begin{thm} \label{Thm2}
$\sigma_{p}(\Delta,h)=\emptyset$.
\end{thm}
\begin{proof}
Suppose that $\Delta x=\alpha x$ for $x\neq \theta$ in $h$. Then, by solving the system of linear equations
\begin{eqnarray*}
\begin{array}{rcl}
x_{0}&=&\alpha x_{0},\\
-x_{0}+x_{1}&=&\alpha x_{1}\\
-x_{1}+x_{2}&=&\alpha x_{2}\\
&\vdots&\\
-x_{n-1}+x_{n}&=&\alpha x_{n}\\
\\ &\vdots&
\end{array}
\end{eqnarray*}

we find that if $x_{n_{0}}$ is the first nonzero entry of the
sequence $x=(x_n)$, then $\alpha=1$. From the equality
$-x_{n_{0}}+x_{n_{0}+1}=\alpha x_{n_{0}+1}$ we have $x_{n_{0}}$ is
zero. This contradicts the fact that $x_{n_{0}}\neq0$, which
completes the proof.
\end{proof}

\begin{thm} \label{Thm3}
$\sigma_{p}(\Delta^{\ast},h^{\ast})=\left\{\alpha
\in\mathbb{C}:\left|1-\alpha\right|< 1\right\}$.
\end{thm}
\begin{proof}
Suppose that $\Delta^{\ast} x = \alpha x$ for $x\neq \theta$ in $h^{\ast}\cong \sigma_{\infty}$. Then, by solving the system of linear equations
\begin{eqnarray*}
\begin{array}{rcl}
x_{0}-x_{1}&=&\alpha x_{0},\\
x_{1}-x_{2}&=&\alpha x_{1},\\
&\vdots&\\
x_{n-1}-x_{n}&=&\alpha x_{n},\\ &\vdots&
\end{array}
\end{eqnarray*}
we observe that $x_{n}=(1-\alpha)^{n} x_{0}$. Therefore, $\stackrel{}{\underset{n}{\sup}}\frac{|x_{0}|}{n}\stackrel{n}{\underset{k=1}{\sum}}|1-\alpha|^{k}<\infty$ if and only if $|1-\alpha|<1$. This step concludes the proof.
\end{proof}

If $T\in B(h)$ with the matrix $A$, then it is known that the
adjoint operator $T^{\ast} :h^{\ast}\rightarrow h^{\ast}$ is defined by the transpose $A^{t}$ of the matrix $A$. It should be noted that the dual space $h^{\ast}$ of $h$ is isometrically isomorphic to the Banach space $\sigma_{\infty}$ of absolutely summable sequences normed by $\|x\|=\stackrel{\infty}{\underset{k=0}{\sum}}k|x_{k}-x_{k+1}|$.

\begin{lem} \cite[p. 59]{go} \label{Lem3.4.} $T$ has a dense range if and only if $T^{\ast}$ is one to one.
\end{lem}
\begin{thm} \label{Thm4}
$\sigma_{r}(\Delta,h)=\sigma_{p}(\Delta^{\ast},h^{\ast})$.
\end{thm}
\begin{proof}
For $\left|1-\alpha\right|< 1$, the operator $\Delta-\alpha I$ is
triangle, so has an inverse. But $\Delta^{\ast}-\alpha I$ is not
one to one by Theorem \ref{Thm3}. Therefore by Lemma
\ref{Lem3.4.}, $\overline{R(\Delta-\alpha I)}\neq h$ and this
step concludes the proof.
\end{proof}

\begin{thm} \label{Thm5}
$\sigma_{c}(\Delta,h)=\left\{\alpha
\in\mathbb{C}:\left|1-\alpha\right|=1\right\}$.
\end{thm}
\begin{proof}
For $\left|1-\alpha\right|< 1$, the operator $\Delta-\alpha I$ is
triangle, so has an inverse but is unbounded. Also $\Delta^{\ast}-\alpha I$ is
one to one by Theorem \ref{Thm3}. By Lemma
\ref{Lem3.4.}, $\overline{R(\Delta-\alpha I)}=h$. Thus, the proof is completed.
\end{proof}

\begin{thm} \label{Thm6}
$A_{3}\sigma(\Delta,h)=B_{3}\sigma(\Delta,h)=C_{3}\sigma(\Delta,h)=\emptyset$.
\end{thm}
\begin{proof}
From Theorem \ref{Thm2} and Table 1.2., $A_{3}\sigma(\Delta,h)=B_{3}\sigma(\Delta,h)=C_{3}\sigma(\Delta,h)=\emptyset$ is observed.
\end{proof}

\begin{thm} \label{Thm7}
$C_{1}\sigma(\Delta,h)=\emptyset$ and $\alpha\in\sigma_{r}(\Delta,h)\cap C_{2}\sigma(\Delta,h) $.
\end{thm}
\begin{proof}
We know $C_{1}\sigma(\Delta,h)\cup C_{2}\sigma(\Delta,h)=\sigma_{r}(\Delta,h)$ from Table 1.2. For $\alpha \in \sigma_{r}(\Delta,h) $, the operator $(\Delta- \alpha I)^{-1}$ is unbounded by Theorem \ref{Thm1}. So $C_{1}\sigma(\Delta,h)=\emptyset$. This completes the proof.
\end{proof}

\begin{thm} \label{Thm8}
The following results hold:
\begin{enumerate}
\item[(a)] $\sigma_{ap}(\Delta,h)=\sigma(\Delta,h)$.
\item[(b)] $\sigma_{\delta}(\Delta,h)=\sigma(\Delta,h)$.
\item[(c)] $\sigma_{co}(\Delta,h)=\left\{\alpha \in \mathbb{C}: \left|1-\alpha\right|<1\right\}$.
\end{enumerate}
\end{thm}
\begin{proof} (a) Since $\sigma_{ap}(\Delta,h)=\sigma(\Delta,h) \backslash C_{1}\sigma(\Delta,h)$ from Table 1.2. and $C_{1}\sigma(\Delta,h)=\emptyset$ by Theorem \ref{Thm7}, we have $\sigma_{ap}(\Delta,h)=\sigma(\Delta,h)$.

(b) Since $\sigma_{\delta}(\Delta,h)=\sigma(\Delta,h) \backslash A_{3}\sigma(\Delta,h)$ from Table 1.2 and  $A_{3}\sigma(\Delta,h)=\emptyset$ by Theorem \ref{Thm6}, we have $\sigma_{\delta}(\Delta,h)=\sigma(\Delta,h)$.

(c) Since the equality $\sigma_{co}(\Delta,h)=C_{1}\sigma(\Delta,h)\cup C_{2}\sigma(\Delta,h)\cup C_{3}\sigma(\Delta,h)$ holds from Table 1.2, we have $\sigma_{co}(\Delta,h)=\left\{\alpha \in \mathbb{C}: \left|1-\alpha\right|<1\right\}$ by Theorems \ref{Thm7} and \ref{Thm8}.
\end{proof}

The next corollary can be obtained from Proposition 2.1.
\begin{cor} The following results hold:
\begin{enumerate}
\item[(a)] $\sigma_{ap}(\Delta^{\ast},\ell_{1})=\sigma(\Delta,h)$.
\item[(b)] $\sigma_{\delta}(\Delta^{\ast},\ell_{1})=\left\{\alpha: \left|\alpha-(2-\delta)^{-1}\right|=(1-\delta)(2-\delta)\right\}\cup E$.
\item[(c)] $\sigma_{p}(\Delta^{\ast},\ell_{1})=\left\{\alpha \in \mathbb{C}: \left|\alpha-(2-\delta)^{-1}\right|<(1-\delta)/(2-\delta)\right\}\cup S$.
\end{enumerate}
\end{cor}

\section{Conclusion}
Hahn \cite{Hahn} defined the space $h$ and gave its some general
properties. Goes and Goes \cite{GoesII} studied the functional
analytic properties of the space $h$. The study on the Hahn sequence
space was initiated by Rao \cite{Rao} with certain specific purpose
in Banach space theory. Also Rao \cite{Rao} emphasized on some matrix
transformations. Rao and Srinivasalu \cite{RaoI} introduced a new
class of sequence space called the semi replete space. Rao and
Subramanian \cite{RaoII} defined the semi Hahn space and proved that
the intersection of all semi Hahn spaces is Hahn space.
Balasubramanian and Pandiarani \cite{balasub} defined the new
sequence space $h(F)$ called the Hahn sequence space of fuzzy
numbers and proved that $\beta-$ and $\gamma-$duals of $h(F)$ is the
Ces\`{a}ro space of the set of all fuzzy bounded sequences. The
sequence space $h$ was introduced by Hahn \cite{Hahn} and Goes and Goes
\cite{GoesII}, and Rao \cite{Rao,RaoI,RaoII}
investigated some properties of the space $h$. Quite recently, Kiri\c{s}ci
\cite{kirisci1} has defined a new Hahn sequence space by using Ces\`{a}ro mean, in
\cite{kirisci2}.

The difference matrix $\Delta$ was used for determining the spectrum
or fine spectrum acting as a linear operator on any of the classical
sequence spaces $c_0$ and $c$, $\ell_1$ and $bv$, $\ell_p$ for
$(1\leq p<\infty)$, respectively in \cite{bafb1}, \cite{Fu1} and
\cite{aafb2}.

As a natural continuation of this paper, one can study the spectrum
and fine spectrum of the Ces\`{a}ro operator, Weighted mean operator
or another known operators in the sequence space $h$.
\section*{Acknowledgement}
We would like to thank Professor Feyzi Ba\c sar, Fatih University,
B\"uy\"uk\c cekmece Campus, 34500 - \.Istanbul, Turkey, for his
careful reading and valuable suggestions on the earlier version of
this paper.

\end{document}